\documentclass[11pt,a4paper]{amsart}
\usepackage[arrow,curve,matrix]{xy}
\usepackage[a4paper,twoside,centering]{geometry}

\title[Mutation classes as derived equivalence classes]
{Mutation classes of certain quivers with potentials as derived
equivalence classes}

\author{Sefi Ladkani}
\address{%
Institut des Hautes \'{E}tudes Scientifiques \\
Le Bois Marie, 35, route de Chartres \\
91440 Bures-sur-Yvette, France}
\email{sefil@ihes.fr}
\urladdr{http://www.ihes.fr/\~{}sefil}

\thanks{This work was supported by a European Postdoctoral Institute (EPDI)
fellowship.}

\DeclareMathOperator{\End}{End}

\newcommand{\cD}{\mathcal{D}}
\newcommand{\eps}{\varepsilon}
\newcommand{\gL}{\Lambda}
\newcommand{\bN}{\mathbb{N}}
\newcommand{\cP}{\mathcal{P}}
\newcommand{\cQ}{\mathcal{Q}}
\newcommand{\cT}{\mathcal{T}}
\newcommand{\bX}{\mathbb{X}}

\theoremstyle{plain}
\newtheorem{theorem}{Theorem}
\newtheorem{prop}{Proposition}[section]
\newtheorem{lemma}[prop]{Lemma}
\newtheorem{cor}[prop]{Corollary}

\theoremstyle{definition}
\newtheorem{remark}[prop]{Remark}

\numberwithin{equation}{section}

\begin{document}

\begin{abstract}
We characterize the marked bordered unpunctured oriented surfaces with
the property that all the Jacobian algebras of the quivers with
potentials arising from their triangulations are derived equivalent.
These are either surfaces of genus $g$ with $b$ boundary components and
one marked point on each component, or the disc with $4$ or $5$ points
on its boundary.

We show that for each such marked surface, all the quivers in the
mutation class have the same number of arrows, and the corresponding
Jacobian algebras constitute a complete derived equivalence class of
finite-dimensional algebras whose members are connected by sequences of
Brenner-Butler tilts. In addition, we provide explicit quivers for each
of these classes.

We consider also $10$ of the $11$ exceptional finite mutation classes
of quivers not arising from triangulations of marked surfaces excluding
the one of the quiver $X_7$, and show that all the finite-dimensional
Jacobian algebras in such class (for suitable choice of potentials) are
derived equivalent only for the classes of the quivers $E_6^{(1,1)}$
and $X_6$.
\end{abstract}

\maketitle

\section{Motivation and Summary of Results}

The Bernstein-Gelfand-Ponomarev reflection~\cite{BGP73} is an operation
on quivers which carries several interpretations. On a combinatorial
level, it takes as inputs an acyclic quiver $Q$ and a vertex $s$ of $Q$
which is a source or a sink in $Q$ and outputs a new quiver $\sigma_s
Q$. On an algebraic level, it gives rise to a derived equivalence
between the path algebras $KQ$ and $K \sigma_sQ$ over any field $K$.
Moreover, when $K$ is algebraically closed, by a result of
Happel~\cite{Happel88}, the path algebras of two acyclic quivers $Q$
and $Q'$ are derived equivalent if and only if $Q'$ can be obtained
from $Q$ by performing a finite number of reflections. It is therefore
plausible to extend the scope of this operation beyond path algebras of
quivers as well as remove the restriction on the vertex to be a sink or
a source.

Indeed, a generalization of the combinatorial aspect of reflection is
given by the quiver mutation introduced by Fomin and
Zelevinsky~\cite{FominZelevinsky02} in their theory of cluster
algebras, allowing to mutate a quiver (without loops and 2-cycles) at
any vertex. Furthermore, Derksen, Weyman and Zelevinsky~\cite{DWZ08}
have developed the theory of quivers with potentials (QP) and their
mutations. The data of a quiver $Q$ and a potential $W$ on it give
rise to the Jacobian algebra $\cP(Q,W)$ which can be seen as a
generalization of the path algebra in the acyclic case.

However, in such generality the algebraic aspect of reflection as a
derived equivalence is usually lost, that is, if $\mu_k(Q,W)$ is the
mutation of $(Q,W)$ at a vertex $k$, the Jacobian algebras $\cP(Q,W)$
and $\cP(\mu_k(Q,W))$ will not be derived equivalent in general.
One remedy to this situation, provided by Keller and
Yang~\cite{KellerYang11}, is to replace the Jacobian algebra by the
Ginzburg dg-algebra $\Gamma(Q,W)$ which is negatively graded and
3-Calabi-Yau, and then the derived categories
of $\Gamma(Q,W)$ and $\Gamma(\mu_k(Q,W))$ are always equivalent.

Another approach is not to replace the Jacobian algebras, but rather
restrict attention to mutation classes of QPs possessing desired
properties regarding derived equivalence. Following this approach, in
this paper we will present mutation classes $\cQ$ of connected quivers
with potentials having the following two properties:
\begin{enumerate}
\item [$(\delta_1)$]
For any $(Q,W) \in \cQ$, the Jacobian algebra $\cP(Q,W)$ is
finite-dimensional;

\item [$(\delta_2)$]
For any $(Q,W), (Q',W') \in \cQ$, the Jacobian algebras
$\cP(Q,W)$ and $\cP(Q',W')$ are derived equivalent.
\end{enumerate}
Note that by~\cite{DWZ08}, it is enough to check condition $(\delta_1)$
for just one member of $\cQ$.

The condition $(\delta_2)$ is quite restrictive. For example, it is
easy to see that while a mutation class of an acyclic connected quiver
$Q$ (necessarily with zero potential) always satisfies the condition
$(\delta_1)$, it will never satisfy the condition $(\delta_2)$ unless
$Q$ has at most two vertices. Indeed, the Jacobian algebras of the QP
in the mutation class of $Q$ are precisely the cluster-tilted algebras
of type $Q$~\cite{BIRSm08,BMR07}. When $Q$ has at most two vertices,
all mutations are BGP reflections so they are derived equivalences,
whereas when $Q$ has three or more vertices there exist cluster-tilted
algebras which are not hereditary and hence of infinite global
dimension~\cite{KellerReiten07}, thus not derived equivalent to the
path algebra $KQ$.

On the other hand, there are mutation classes satisfying condition
$(\delta_2)$ but not $(\delta_1)$.
Indeed, by~\cite{KellerYang11} the property that the
Ginzburg dg-algebra has its cohomology concentrated only in degree zero
(and hence is quasi-isomorphic to the Jacobian algebra)
is preserved under QP mutation. Mutation classes with this property
thus satisfy $(\delta_2)$, but as their Jacobian algebras are 3-Calabi-Yau,
condition $(\delta_1)$ does not hold.

Since we deal with mutation classes, it is natural to ask when a single
mutation of QP leads to derived equivalence of the corresponding
Jacobian algebras. Possible candidates for tilting complexes, studied
in this context by Vitoria~\cite{Vitoria09} and Keller-Yang~\cite[\S
6]{KellerYang11}, take the following form. For an algebra $\gL$ given
by a quiver with relations and a vertex $k$ without loops, consider the
complexes
\begin{align*}
T^-_k(\gL) = \bigl( P_k \xrightarrow{f} \bigoplus_{j \to k} P_j \bigr) \oplus
\bigl(\bigoplus_{i \neq k} P_i \bigr)
&, &
T^+_k(\gL) = \bigl( \bigoplus_{k \to j} P_j \xrightarrow{g} P_k \bigr) \oplus
\bigl(\bigoplus_{i \neq k} P_i \bigr)
\end{align*}
where $P_i$ denotes the indecomposable projective corresponding to $i$,
the map $f$ (respectively, $g$) is induced by all the arrows ending
(respectively, starting) at $k$, and the terms $P_i$ for $i \neq k$ lie
in degree $0$. Each of these complexes is not always tilting, but when
it is, it induces a derived equivalence $\cD(\gL) \xrightarrow{\sim}
\cD(\End T^-_k(\gL))$ (or $\cD(\gL) \xrightarrow{\sim} \cD(\End
T^+_k(\gL))$) which generalizes the BGP reflection functor and forms a
specific instance of a perverse Morita equivalence~\cite{Rouquier06}.
In this case we denote the endomorphism algebra by $\mu^-_k(\gL)$
(resp.\ $\mu^+_k(\gL)$) and call it the \emph{negative} (resp.\
\emph{positive}) mutation of the algebra $\gL$, see~\cite{Ladkani10}.
Note that it might well happen that none of the algebra mutations is
defined, or that both of them are defined but not isomorphic.

Given a Jacobian algebra $\cP(Q,W)$ and a vertex $k$, there are now two
notions of mutation that we may consider. The first is mutation of
quivers with potentials leading to the Jacobian algebra
$\cP(\mu_k(Q,W))$, whereas the second is given by the negative and
positive algebra mutations of $\cP(Q,W)$. Roughly speaking, the
mutation is \emph{good} if these two notions are compatible. More
precisely, the mutation of $(Q,W)$ at the vertex $k$ is \emph{good} if
$\mu^-_k(\cP(Q,W)) \simeq \cP(\mu_k(Q,W))$ or $\mu^+_k(\cP(Q,W)) \simeq
\cP(\mu_k(Q,W))$.

By definition, a good mutation implies the derived equivalence of the
corresponding Jacobian algebras, known in the physics literature as
Seiberg duality. Hence one is motivated to consider the following
property which implies the condition $(\delta_2)$.
\begin{enumerate}
\item [$(\delta_3)$]
For any $(Q,W) \in \cQ$, the mutation at any vertex $k$ of $Q$ is good.
Furthermore, if both algebra mutations of $\cP(Q,W)$ at $k$ are defined,
they are isomorphic.
\end{enumerate}

It will turn out that all the mutation classes possessing properties
$(\delta_1)$ and $(\delta_2)$ which we will present have also the
stronger property $(\delta_3)$. Note that for these classes an answer
to Question~12.2 in~\cite{DWZ08} can be given in a very explicit way,
namely at any vertex $k$ the (unique) algebra mutation of $\cP(Q,W)$
coincides with $\cP(\mu_k(Q,W))$.

Motivated by algorithmic applications, e.g.\ \cite[\S 5.3]{Ladkani10},
we consider the additional finiteness condition:
\begin{enumerate}
\item [$(\delta_4)$]
$\cQ$ consists of a finite number of quivers.
\end{enumerate}

According to Felikson, Shapiro and Tumarkin~\cite{FST08b}, the
connected quivers whose mutation class is finite are either those
arising from triangulations of bordered oriented surfaces with marked
points as introduced by Fomin, Shapiro and Thurston~\cite{FST08}, or
they are mutation equivalent to one of 11 exceptional quivers, or they
are acyclic with $2$ vertices and $r \geq 3$ arrows between them.

\subsection{Mutation classes from triangulations of bordered surfaces
with marked points} \label{sec:surface}

For quivers arising from such triangulations, potentials have been
defined by Labardini-Fragoso~\cite{Labardini09}. In this paper we
further restrict our attention to the case of \emph{no punctures}, that
is, the marked points lie on the boundary of the surface. The
associated potentials are then sums of oriented $3$-cycles
(\emph{triangles}) and the resulting Jacobian algebras are the gentle
algebras studied by Assem, Br\"{u}stle, Charbonneau-Jodoin and
Plamondon~\cite{ABCP10}.

We start by characterizing the configurations of marked points yielding
mutations classes satisfying the property $(\delta_2)$.

\begin{theorem} \label{t:SM}
Let $S$ be a surface with boundary $\partial S$ and $M \subset \partial
S$ a finite set of marked points with at least one point in any
connected component of $\partial S$.

Then the mutation class of quivers with potentials corresponding to the
triangulations of $(S,M)$ satisfies condition $(\delta_2)$ if and only
if either:
\begin{itemize}
\item
$M$ contains exactly one point from each connected component of
$\partial S$, or

\item
$S$ is a disc and $M$ consists of $4$ or $5$ points.
\end{itemize}
\end{theorem}

In view of this theorem, any $g \geq 0$ and $b \geq 1$ such that $(g,b)
\neq (0,1)$ give rise to a mutation class of QPs which we will denote
by $\cQ_{g,b}$ having the required properties $(\delta_1)$ and
$(\delta_2)$. Namely, take $S$ to be a bordered surface of genus $g$
with $b$ boundary components, $M$ to be a set of $b$ points containing
one point from each boundary component and consider all the quivers
with potentials arising from the triangulations of $(S,M)$. Note that
the case $(0,1)$ corresponding to the disc is excluded as it has no
triangulations.

We denote by $\cT_{g,b}$ the class of Jacobian algebras of the QPs in
$\cQ_{g,b}$. As these are finite-dimensional gentle algebras, one can
consider the derived invariant developed by Avella-Alaminos and
Geiss~\cite{AvellaAlaminosGeiss08} given as a function $\bN^2 \to \bN$
which we write as a finite sum $\sum_{i=1}^{r} c_i (n_i,m_i)$ where
$c_i>0$ and $(n_1,m_1), \dots, (n_r,m_r)$ are distinct elements of
$\bN^2$.

Further properties of the mutation class $\cQ_{g,b}$ and the
corresponding algebras $\cT_{g,b}$ are elaborated in the next theorem.

\begin{theorem} \label{t:gb}
Let $(g,b) \neq (0,1)$ and let $\cQ_{g,b}$ be the class of quivers with
potentials arising from triangulations of a bordered surface of genus
$g$ with $b$ boundary components and $b$ marked points, one at each
boundary component. Let $\cT_{g,b}$ be the corresponding class of
Jacobian algebras.
\begin{enumerate}
\renewcommand{\theenumi}{\alph{enumi}}
\item
$\cQ_{g,b}$ is a mutation class satisfying conditions
$(\delta_1)$, $(\delta_2)$, $(\delta_3)$ and $(\delta_4)$.

\item
Any quiver in $\cQ_{g,b}$ has $n$ vertices, $e$ arrows and $t$ triangles,
where
\begin{align*}
n = 6(g-1)+4b &,& e=2n-b=12(g-1)+7b &,& t=4(g-1)+2b .
\end{align*}

\item
For any algebra $\gL \in \cT_{g,b}$,
\begin{enumerate}
\renewcommand{\theenumii}{\roman{enumii}}
\item
the determinant of its Cartan matrix is $2^t$,
\item
its Avella-Alaminos-Geiss invariant is $b(1,1)+t(0,3)$.
\end{enumerate}

\item
Any two algebras in $\cT_{g,b}$ are connected by a sequence of algebra
mutations, that is, if $\gL, \gL' \in \cT_{g,b}$ then there exist
algebras $\gL = \gL_0, \gL_1, \dots, \gL_m = \gL'$ in $\cT_{g,b}$ and a
sequence of vertices $k_1, k_2, \dots, k_m$ such that for any $1 \leq j
\leq m$ we have $\gL_j = \mu^-_{k_j}(\gL_{j-1})$ or $\gL_j =
\mu^+_{k_j}(\gL_{j-1})$.

\item
$\cT_{g,b}$ is closed under derived equivalence, that is, if $\gL \in
\cT_{g,b}$ and $\gL'$ is derived equivalent to $\gL$, then also $\gL'
\in \cT_{g,b}$.

\item
When $(g,b) \neq (g',b')$, the classes $\cQ_{g,b}$ and $\cQ_{g',b'}$
(as well as $\cT_{g,b}$ and $\cT_{g',b'}$) are disjoint.
\end{enumerate}
\end{theorem}

In particular we see that for any even positive integer $n$ there
exists at least one mutation class of QPs with the properties
$(\delta_1)$, $(\delta_2)$, $(\delta_3)$, $(\delta_4)$ whose quivers
have $n$ vertices. For small number of vertices, details for these
classes are given in Table~\ref{tab:Qgb} and representative quivers
from each class are shown in Figure~\ref{fig:Qgb}. To make our results
more explicit we will outline a procedure to draw these quivers for any
$(g,b)$ in Section~\ref{sec:quivers}.

\begin{table}
\begin{center}
\begin{tabular}{|c||c|c|c||c|}
\hline
$(g,b)$ & Vertices & Arrows & Triangles & Size of $\cQ_{g,b}$ \\
\hline (0,2) & 2 & 2 & 0 & 1\\
(1,1) & 4 & 7 & 2 & 1 \\
(0,3) & 6 & 9 & 2 & 6 \\
(1,2) & 8 & 14 & 4 & 56 \\
(0,4) & 10 & 16 & 4 & 140 \\
(2,1) & 10 & 19 & 6 & 105 \\
(1,3) & 12 & 21 & 6 & 3236 \\
\hline
\end{tabular}
\end{center}
\caption{The numbers of vertices, arrows and triangles for quivers
in $\cQ_{g,b}$ and the size of $\cQ_{g,b}$ for small values of $(g,b)$.}
\label{tab:Qgb}
\end{table}

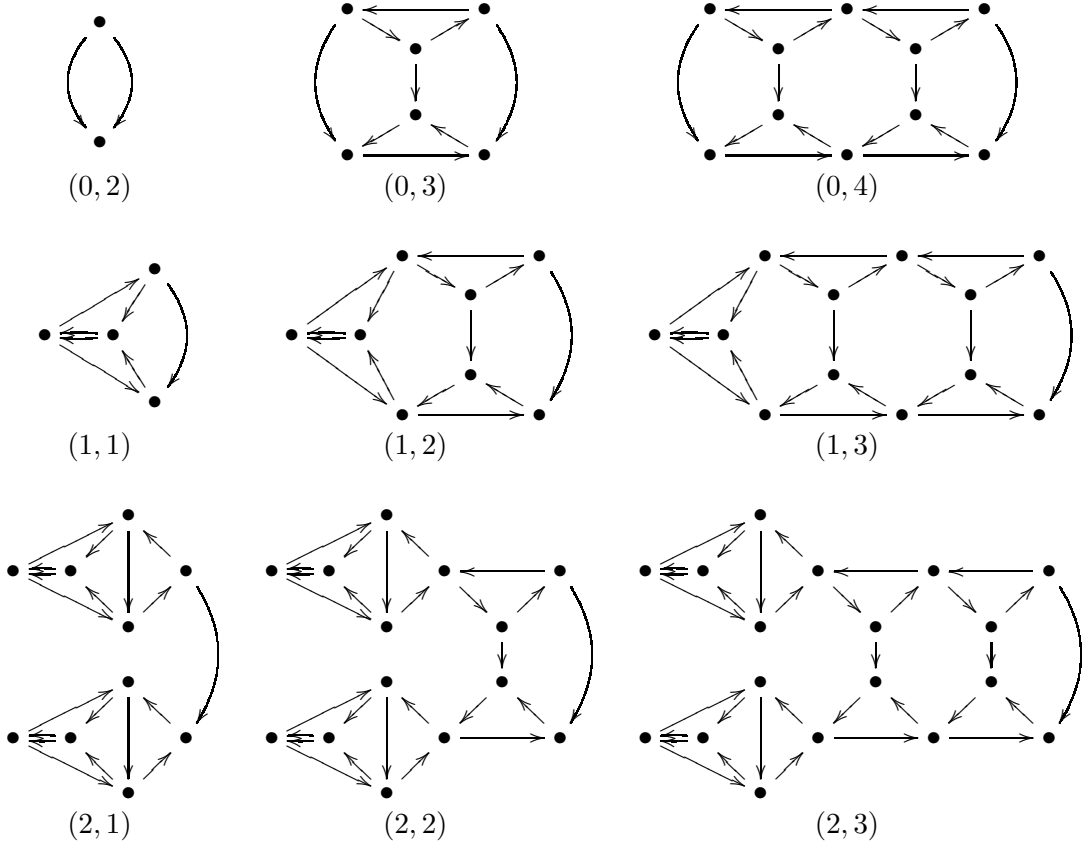
\begin{figure}
\[
\begin{array}{ccc}
\begin{array}{c}
\xymatrix@=0.3pc{
{\bullet} \ar@/_1pc/[dddd] \ar@/^1pc/[dddd] \\ \\ \\ \\
{\bullet}
}
\end{array}
&
\begin{array}{c}
\xymatrix@=0.3pc{
{\bullet} \ar@/_1pc/[dddd] \ar[drr]
&& && {\bullet} \ar[llll] \ar@/^1pc/[dddd] \\
&& {\bullet} \ar[dd] \ar[urr] \\ \\
&& {\bullet} \ar[dll] \\
{\bullet} \ar[rrrr] && && {\bullet} \ar[ull]
}
\end{array}
&
\begin{array}{c}
\xymatrix@=0.3pc{
{\bullet} \ar@/_1pc/[dddd] \ar[drr]
&& && {\bullet} \ar[drr] \ar[llll]
&& && {\bullet} \ar[llll] \ar@/^1pc/[dddd] \\
&& {\bullet} \ar[dd] \ar[urr]
&& && {\bullet} \ar[dd] \ar[urr] \\ \\
&& {\bullet} \ar[dll]
&& && {\bullet} \ar[dll] \\
{\bullet} \ar[rrrr]
&& && {\bullet} \ar[ull] \ar[rrrr]
&& && {\bullet} \ar[ull]
}
\end{array}
\\
(0,2) & (0,3) & (0,4)
\\ \\
\begin{array}{c}
\xymatrix@=0.3pc{
&&& {\bullet} \ar[ddl] \ar@/^1pc/[dddd] \\ \\
{\bullet} \ar[uurrr] \ar[ddrrr]
&& {\bullet} \ar@/_0.1pc/[ll] \ar@/^0.1pc/[ll] \\ \\
&&& {\bullet} \ar[uul]
}
\end{array}
&
\begin{array}{c}
\xymatrix@=0.3pc{
&&& {\bullet} \ar[ddl] \ar[drr]
&& && {\bullet} \ar[llll] \ar@/^1pc/[dddd] \\
&&& && {\bullet} \ar[dd] \ar[urr] \\
{\bullet} \ar[uurrr] \ar[ddrrr]
&& {\bullet} \ar@/_0.1pc/[ll] \ar@/^0.1pc/[ll] \\
&&& && {\bullet} \ar[dll] \\
&&& {\bullet} \ar[uul] \ar[rrrr]
&& && {\bullet} \ar[ull]
}
\end{array}
&
\begin{array}{c}
\xymatrix@=0.3pc{
&&& {\bullet} \ar[ddl] \ar[drr]
&& && {\bullet} \ar[llll] \ar[drr]
&& && {\bullet} \ar[llll] \ar@/^1pc/[dddd] \\
&&& && {\bullet} \ar[dd] \ar[urr]
&& && {\bullet} \ar[dd] \ar[urr] \\
{\bullet} \ar[uurrr] \ar[ddrrr]
&& {\bullet} \ar@/_0.1pc/[ll] \ar@/^0.1pc/[ll] \\
&&& && {\bullet} \ar[dll] && && {\bullet} \ar[dll] \\
&&& {\bullet} \ar[uul] \ar[rrrr]
&& && {\bullet} \ar[ull] \ar[rrrr]
&& && {\bullet} \ar[ull]
}
\end{array}
\\
(1,1) & (1,2) & (1,3)
\\ \\
\begin{array}{c}
\xymatrix@=0.8pc{
&& {\bullet} \ar[dl] \ar[dd] \\
{\bullet} \ar[urr] \ar[drr]
& {\bullet} \ar@/_0.1pc/[l] \ar@/^0.1pc/[l]
& & {\bullet} \ar[ul] \ar@/^1pc/[ddd] \\
&& {\bullet} \ar[ur] \ar[ul] \\
&& {\bullet} \ar[dl] \ar[dd] \\
{\bullet} \ar[urr] \ar[drr]
& {\bullet} \ar@/_0.1pc/[l] \ar@/^0.1pc/[l]
& & {\bullet} \ar[ul] \\
&& {\bullet} \ar[ur] \ar[ul]
}
\end{array}
&
\begin{array}{c}
\xymatrix@=0.8pc{
&& {\bullet} \ar[dl] \ar[dd] \\
{\bullet} \ar[urr] \ar[drr]
& {\bullet} \ar@/_0.1pc/[l] \ar@/^0.1pc/[l]
& & {\bullet} \ar[ul] \ar[dr]
&& {\bullet} \ar[ll] \ar@/^1pc/[ddd] \\
&& {\bullet} \ar[ur] \ar[ul] && {\bullet} \ar[ur] \ar[d] \\
&& {\bullet} \ar[dl] \ar[dd] && {\bullet} \ar[dl] \\
{\bullet} \ar[urr] \ar[drr]
& {\bullet} \ar@/_0.1pc/[l] \ar@/^0.1pc/[l]
& & {\bullet} \ar[ul] \ar[rr] && {\bullet} \ar[ul] \\
&& {\bullet} \ar[ur] \ar[ul]
}
\end{array}
&
\begin{array}{c}
\xymatrix@=0.8pc{
&& {\bullet} \ar[dl] \ar[dd] \\
{\bullet} \ar[urr] \ar[drr]
& {\bullet} \ar@/_0.1pc/[l] \ar@/^0.1pc/[l]
& & {\bullet} \ar[ul] \ar[dr]
&& {\bullet} \ar[ll] \ar[dr] && {\bullet} \ar[ll] \ar@/^1pc/[ddd] \\
&& {\bullet} \ar[ur] \ar[ul] && {\bullet} \ar[ur] \ar[d]
&& {\bullet} \ar[ur] \ar[d] \\
&& {\bullet} \ar[dl] \ar[dd] && {\bullet} \ar[dl]
&& {\bullet} \ar[dl] \\
{\bullet} \ar[urr] \ar[drr]
& {\bullet} \ar@/_0.1pc/[l] \ar@/^0.1pc/[l]
& & {\bullet} \ar[ul] \ar[rr] && {\bullet} \ar[ul] \ar[rr]
&& {\bullet} \ar[ul] \\
&& {\bullet} \ar[ur] \ar[ul]
}
\end{array}
\\
(2,1) & (2,2) & (2,3)
\end{array}
\]
\caption{Representative quivers in each $\cQ_{g,b}$ for small values of
$(g,b)$. The potentials are sums of the oriented 3-cycles (chosen such
that any arrow belongs to at most one cycle).}
\label{fig:Qgb}
\end{figure}

The theorem implies that each of the classes $\cQ_{g,b}$ (and
$\cT_{g,b}$) enjoys two remarkable properties. First, all the quivers
in $\cQ_{g,b}$ have the same number of arrows. This is quite rare for
an arbitrary mutation class (which is then necessarily finite). Second,
not only that all the algebras in $\cT_{g,b}$ are derived equivalent,
but $\cT_{g,b}$ is closed under derived equivalence and hence
constitutes an entire derived equivalence class of algebras which can
be explicitly described using QP and their mutations.

Explicit descriptions of all the algebras derived equivalent to a given
algebra $\gL$ are quite rare since by Rickard theorem~\cite{Rickard89}
any endomorphism ring of a tilting complex over $\gL$ gives rise to an
algebra derived equivalent to $\gL$, and it is usually hard to control
all the possible tilting complexes. Some notable instances of algebras
$\gL$ where such explicit descriptions have been obtained include the
path algebras of quivers of Dynkin types $A$ \cite{AssemHappel81} and
$D$ \cite{Keller91} as well as affine type $\tilde{A}$
\cite{AssemSkowronski87}. Other instances are the classes of Brauer
tree algebras with fixed numerical parameters (number of edges,
multiplicity of the exceptional vertex) which are closed under derived
equivalence and moreover any two algebras in a class can be connected
by a sequence of algebra mutations, see~\cite[\S5]{KoenigZimmermann98}.

\subsection{Exceptional mutation classes}
\label{sec:exceptional}

We now turn our attention to the $11$ exceptional quivers whose
mutation classes are finite. Of these, $E_6$, $E_7$, $E_8$,
$\tilde{E}_6$, $\tilde{E}_7$ and $\tilde{E}_8$ are acyclic and hence
their mutation classes will not satisfy condition $(\delta_2)$. In
order to deal with $E_6^{(1,1)}$, $E_7^{(1,1)}$ and $E_8^{(1,1)}$, we
consider the more general QP given by the quiver
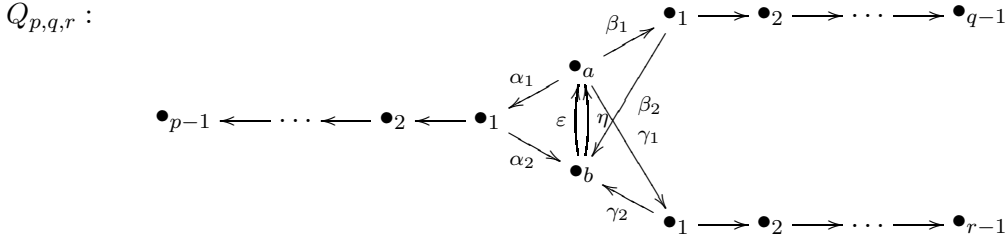
\begin{align*}
Q_{p,q,r}: & &
\xymatrix@=0.5pc{
&& && && && && {\bullet_1}
\ar[dddll]^{\beta_2} \ar[rr] && {\bullet_2} \ar[rr]
&& {\dots} \ar[rr] && {\bullet_{q-1}} \\
&& && && && {\bullet_a}
\ar[urr]^{\beta_1} \ar[dddrr]^{\gamma_1} \ar[dll]_{\alpha_1} \\
{\bullet_{p-1}} && {\dots} \ar[ll] &&
{\bullet_2} \ar[ll] && {\bullet_1} \ar[ll] \ar[drr]_{\alpha_2} \\
&& && && && {\bullet_b} \ar@/^0.2pc/[uu]^{\eps} \ar@/_0.2pc/[uu]_{\eta} \\
&& && && && && {\bullet_1}
\ar[ull]^{\gamma_2} \ar[rr] && {\bullet_2} \ar[rr]
&& {\dots} \ar[rr] && {\bullet_{r-1}}
}
\end{align*}
and the potential $W_{p,q,r} = \eps(\alpha_1 \alpha_2 + \beta_1
\beta_2) + \eta(\alpha_1 \alpha_2 + \gamma_1 \gamma_2)$ for some $p,q,r
\geq 2$. The quivers $E_6^{(1,1)}$, $E_7^{(1,1)}$, $E_8^{(1,1)}$
coincide with the quivers $Q_{3,3,3}$, $Q_{2,4,4}$, $Q_{2,3,6}$,
respectively.

Similar diagrams appear in singularity theory as Coxeter-Dynkin
diagrams~\cite[\S 3.9]{LenzingdelaPena06}. In fact, the Jacobian
algebra $\cP(Q_{p,q,r},W_{p,q,r})$ can be realized as the endomorphism
algebra of a cluster-tilting object in the cluster category of the
weighted projective line $\bX_{p,q,r}$~\cite{BKL10}. As already
observed by Barot and Geiss~\cite{BarotGeiss09}, the quivers
$E_6^{(1,1)}$, $E_7^{(1,1)}$, $E_8^{(1,1)}$ can thus be realized by the
tubular cluster categories corresponding to the tubular weights
$(3,3,3)$, $(2,4,4)$, $(2,3,6)$ respectively. Moreover, the quiver
$E_6^{(1,1)}$ which turns out to be of particular interest with regard
to derived equivalence is also realized by the stable category of
modules over the preprojective algebra of Dynkin type $D_4$ studied by
Geiss-Leclerc-Schr\"{o}er~\cite{GLS06}.

\begin{theorem}
Let $p,q,r \geq 2$ and let $\cQ_{p,q,r}$ be the mutation class of
$(Q_{p,q,r}, W_{p,q,r})$ defined above. Then:

\begin{enumerate}
\renewcommand{\theenumi}{\alph{enumi}}
\item
$\cQ_{p,q,r}$ satisfies condition $(\delta_2)$ if and only if $(p,q,r)=(3,3,3)$.

\item
The class $\cQ_{3,3,3}$ corresponding to the quiver $E_6^{(1,1)}$
consists of $49$ QPs and satisfies conditions $(\delta_1)$,
$(\delta_2)$, $(\delta_3)$, $(\delta_4)$.

\item
Furthermore, the Jacobian algebra $\gL = \cP(Q,W)$ of any $(Q,W) \in
\cQ_{3,3,3}$ has the following properties:
\begin{enumerate}
\renewcommand{\theenumii}{\roman{enumii}}
\item
For any vertex $k$, one and only one of the algebra mutations
$\mu^-_k(\gL)$, $\mu^+_k(\gL)$ is defined;

\item
The determinant of the Cartan matrix of $\gL$ is $4$ and its Coxeter
polynomial is $(x^2+1)^4$.
\end{enumerate}
\end{enumerate}
\end{theorem}

Finally we have to deal with the quivers $X_6$ and $X_7$ discovered by
Derksen and Owen \cite{DerksenOwen08}. While for $X_7$ we could not find
a potential whose Jacobian algebra is finite-dimensional, for $X_6$ we
have the following result.

\begin{theorem}
Consider the quiver $X_6$ and the potential $W_6$ as given below:
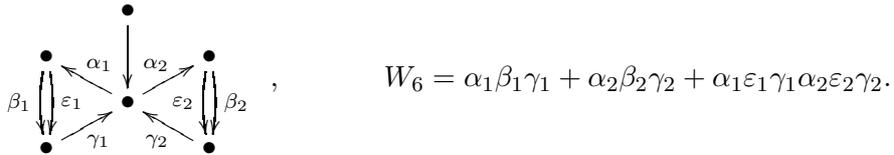
\begin{align*}
\begin{array}{c}
\xymatrix@=0.5pc{
&& {\bullet} \ar[dd] \\
{\bullet} \ar@/^0.2pc/[dd]^{\eps_1} \ar@/_0.2pc/[dd]_{\beta_1} && &&
{\bullet} \ar@/^0.2pc/[dd]^{\beta_2} \ar@/_0.2pc/[dd]_{\eps_2} \\
&& {\bullet} \ar[ull]_{\alpha_1} \ar[urr]^{\alpha_2} \\
{\bullet} \ar[urr]_{\gamma_1} && && {\bullet} \ar[ull]^{\gamma_2}
}
\end{array}
&, &
W_6 = \alpha_1 \beta_1 \gamma_1 + \alpha_2 \beta_2\gamma_2 +\alpha_1
\eps_1 \gamma_1 \alpha_2 \eps_2 \gamma_2 .
\end{align*}
Then the mutation class of $(X_6,W_6)$ consists of $5$ QPs and
satisfies conditions $(\delta_1)$, $(\delta_2)$, $(\delta_3)$,
$(\delta_4)$. The Cartan determinant of the corresponding Jacobian
algebras is $4$ and their Coxeter polynomial is $(x-1)^6$.
\end{theorem}

\subsection*{Acknowledgements}
I would like to thank Maxim Kontsevich and Frol Zapolsky for useful
discussions.

\section{On the proofs}

\subsection{Good mutations for gentle Jacobian algebras}

Consider a gentle algebra $\gL = KQ/I$ given as a quiver with relations
over an algebraically closed field. Since the only relations are
zero-relations (of length $2$), the algebra $\gL$ has a basis
consisting of all the non-zero paths.

A \emph{maximal non-zero path} (known also as a non-trivial permitted
thread) is a path $\alpha_1 \alpha_2 \dots \alpha_m$ between some
vertices $i$ and $j$ which is non-zero in $\gL$ and is maximal with
this property, that is, for any arrow $\beta$ ending at $i$ and arrow
$\gamma$ starting at $j$ we have $\beta \alpha_1 = 0$ and $\alpha_m
\gamma = 0$ in $\gL$. Any arrow is contained in a unique maximal
non-zero path.

Similarly, a \emph{maximal zero path} (known also as a non-trivial
forbidden thread) is a path $\alpha_1 \alpha_2 \dots \alpha_m$ in $Q$
such that $\alpha_s \alpha_{s+1} = 0$ in $\gL$ for all $1 \leq s < m$
which is maximal with this property. For further details we refer the
reader to~\cite{AvellaAlaminosGeiss08}.

We start by characterizing when algebra mutations are defined in terms
of maximal paths. In addition, the next proposition implies that
algebra mutations of gentle algebras without loops coincide with
Brenner-Butler tilts~\cite{BrennerButler80}.

\begin{prop} \label{p:algmut}
Let $\gL=KQ/I$ be gentle and let $k$ be a vertex of $Q$ without loops.
\begin{enumerate}
\renewcommand{\theenumi}{\alph{enumi}}
\item
$\mu^-_k(\gL)$ is defined if and only if no maximal non-zero path
starts at $k$.

\item
$\mu^+_k(\gL)$ is defined if and only if no maximal non-zero path ends
at $k$.

\item
$\mu^-_k(\gL)$ is defined if and only if the Brenner-Butler tilting
module of $\gL$ at $k$ is defined.

\item
$\mu^+_k(\gL)$ is defined if and only if the Brenner-Butler tilting
module of $\gL^{op}$ at $k$ is defined.
\end{enumerate}
\end{prop}
\begin{proof}
Follows from~\cite[Prop.~2.3]{Ladkani10}.
\end{proof}

Let $(Q,W)$ be a QP arising from a triangulation of a bordered
unpunctured oriented surface, so that its Jacobian algebra is gentle.
The next proposition characterizes the vertices $k$ for which mutations
are good in terms of their neighborhoods. The \emph{neighborhood} of a
vertex $k$ is the full subquiver on the set consisting of $k$ and all
vertices which are targets of arrows starting at $k$ or sources of
arrows ending at $k$. By a \emph{triangle} in a quiver arising from a
triangulation we mean an oriented $3$-cycle bound by radical square
zero relations. Such triangles are in bijection with the internal
triangles of the triangulation.

\begin{prop}
\label{p:goodmut} Let $(Q,W)$ be a QP arising from a triangulation of a
bordered unpunctured oriented surface. Then:
\begin{enumerate}
\renewcommand{\theenumi}{\alph{enumi}}
\item
The mutation of $(Q,W)$ at a vertex is good if and only if its
neighborhood is not one of the following
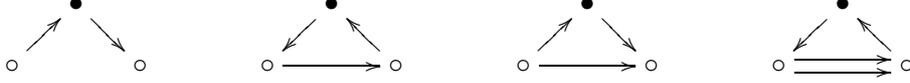
\begin{align*}
\xymatrix@=1pc{
& {\bullet} \ar[dr] \\
{\circ} \ar[ur] & & {\circ}
}
& &
\xymatrix@=1pc{
& {\bullet} \ar[dl] \\
{\circ} \ar[rr] & & {\circ} \ar[ul]
}
& &
\xymatrix@=1pc{
& {\bullet} \ar[dr] \\
{\circ} \ar[rr] \ar[ur] & & {\circ}
}
& &
\xymatrix@=1pc{
& {\bullet} \ar[dl] \\
{\circ} \ar@<-0.5ex>[rr] \ar@<0.5ex>[rr] & & {\circ} \ar[ul]
}
\end{align*}
where $\bullet$ denotes the vertex.

\item
If the mutation of $(Q,W)$ at a vertex $k$ is good and both the
negative and positive algebra mutations of $\cP(Q,W)$ at $k$ are
defined, then they are isomorphic.

\item
A good mutation preserves the numbers of arrows and triangles in the
quiver.
\end{enumerate}
\end{prop}
\begin{proof}

To assess whether a mutation is good or not, \emph{a-priori} one needs
to check that an algebra mutation of $\cP(Q,W)$ at $k$ is defined as
well as to verify that it is isomorphic to $\cP(\mu_k(Q,W))$. Compared
to the former, the latter verification is much harder, so we have
developed criteria and algorithms to test for good mutations using only
the data whether the relevant algebra mutations are defined or
not~\cite[\S 5]{Ladkani10}, building on the notion of $\cD$-split
sequences of Hu and Xi~\cite{HuXi08}.

Note that these criteria were formulated for cluster-tilting objects in
2-Calabi-Yau categories, but by~\cite{Amiot09} and~\cite{BIRSm08}, the
Jacobian algebras $\cP(Q,W)$ and $\cP(\mu_k(Q,W))$ can be realized as
endomorphism algebras of neighboring cluster-tilting objects in a
2-Calabi-Yau triangulated category.

In the case of QP arising from triangulations of unpunctured bordered
surfaces, the only relations lie in oriented $3$-cycles, so in order to
determine whether an algebra mutation is defined at some vertex, it is
enough to consider its neighborhood. As the number of possible
neighborhoods is finite, they can be checked on a computer. Examples of
similar checks can be seen in~\cite[\S 3]{BHL10}.
\end{proof}

\begin{prop}
\label{p:badmut} Let $(Q,W)$ be a QP arising from a triangulation of a
bordered unpunctured oriented surface. If the mutation at the vertex
$k$ is not good, then the Jacobian algebras $\cP(Q,W)$ and
$\cP(\mu_k(Q,W))$ are not derived equivalent.
\end{prop}
\begin{proof}
The determinant of the Cartan matrix of such gentle Jacobian algebra is
$2^t$ where $t$ denotes the number of $3$-cycles with radical square
zero relations, as computed by Holm~\cite{Holm05}. The result now
follows by observing that mutations which are not good change the
number of triangles in the quiver by $1$.
\end{proof}

These results allow for a description of the derived equivalence
classes of the gentle algebras arising from triangulations of a given
marked unpunctured surface in terms of the properties of the
corresponding triangulations, generalizing the derived equivalence
classifications of cluster-tilted algebras of Dynkin type
$A$~\cite{BuanVatne08} and affine type $\tilde{A}$~\cite{Bastian09}.
This is a subject of further investigations.

\subsection{A necessary condition for derived equivalence}
\label{ssec:necessary}

Let $S$ be a surface of genus $g \geq 0$ with $b \geq 1$ boundary
components. Let $M$ be a set of marked points on the boundary, with at
least one point on each component.

\begin{lemma}
\label{l:triangE} Assume that $(g,b) \neq (0,1)$. Then there exists a
triangulation of $(S,M)$ containing, for any boundary component $C$
with marked points $A_1, A_2, \dots, A_m$, two distinct arcs $i$ and
$j$ starting at $A_1$ and $A_m$ respectively having a common endpoint
$E$ (which might be on $C$) as in the following picture.
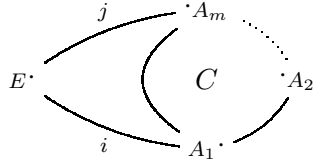
\[
\xymatrix@=1pc{
& & & {\cdot_{A_m}} \ar@{-}@/_2pc/[dd] \\
{_E\cdot} \ar@{-}@/_/[drrr]_i \ar@{-}@/^/[urrr]^j & & & C
& {\cdot_{A_2}} \ar@{.}@/_/[ul] \\
& & & {_{A_1}\cdot} \ar@{-}@/_/[ur]
}
\]
\end{lemma}
\begin{proof}
By induction on the number of marked points in $M$. When this number is
minimal, that is, $M$ contains exactly one point from each component,
such triangulations can be explicitly constructed, see
Section~\ref{sec:quivers}.

Suppose we have constructed such triangulation for $M$ and let $M' = M
\cup \{A_{m+1}\}$ where $A_{m+1}$ is a new marked point on a boundary
component $C$. Then we obtain a triangulation for $M'$ with the
required property by adding the arc $k$ connecting $E$ and $A_{m+1}$ as
in the following picture.
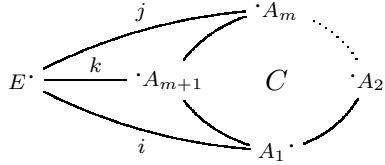
\begin{equation}
\label{e:ijk}
\xymatrix@=1pc{
& & & {\cdot_{A_m}} \ar@{-}@/_/[dl] \\
{_E\cdot} \ar@{-}@/_/[drrr]_i \ar@{-}@/^/[urrr]^j \ar@{-}[rr]^k & &
{\cdot_{A_{m+1}}} \ar@{-}@/_/[dr] & C & {\cdot_{A_2}} \ar@{.}@/_/[ul] \\
& & & {_{A_1}\cdot} \ar@{-}@/_/[ur]
}
\end{equation}
\end{proof}

\begin{prop} \label{p:SMd2}
Assume that $(g,b) \neq (0,1)$. If the mutation class of the QP arising
from the triangulations of $(S,M)$ satisfies condition $(\delta_2)$,
then $M$ contains exactly one point from each boundary component of
$S$.
\end{prop}
\begin{proof}
Assume that $M$ has $m+1$ points $A_1, \dots, A_m, A_{m+1}$ on a
boundary component $C$ for some $m \geq 1$. By applying
Lemma~\ref{l:triangE} for $M \setminus \{A_{m+1}\}$ and then the
inductive step in its proof, we get a triangulation having the arcs $i,
j, k$ as in~\eqref{e:ijk}. Let $(Q,W)$ denote the quiver with potential
corresponding to this triangulation. Then the neighborhood of $k$ in
$Q$ is one of
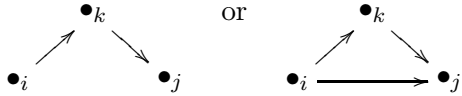
\[
\xymatrix@=1pc{
& {\bullet_k} \ar[dr] \\
{\bullet_i} \ar[ur] & & {\bullet_j}
}
\quad \text{or}
\quad
\xymatrix@=1pc{
& {\bullet_k} \ar[dr] \\
{\bullet_i} \ar[ur] \ar[rr] & & {\bullet_j}
}
\]
hence by Proposition~\ref{p:goodmut} the mutation of $(Q,W)$ at $k$ is
not good. By Proposition~\ref{p:badmut}, the Jacobian algebras
$\cP(Q,W)$ and $\cP(\mu_k(Q,W))$ are not derived equivalent.
\end{proof}

\begin{remark} \label{rem:An}
If $(g,b)=(0,1)$, that is, $S$ is a disc, then $M$ has at least $4$
points on its boundary. Denoting the number of points by $n+3$ for some
$n \geq 1$, it is well known that the quivers arising from the
triangulations of $(S,M)$ are precisely those in the mutation class of
the Dynkin quiver $A_n$. Such class will satisfy condition $(\delta_2)$
only for $n=1,2$.
\end{remark}

\subsection{Sufficiency and further properties}

We fix $(g,b) \neq (0,1)$. Recall that $\cQ_{g,b}$ denotes the mutation
class consisting of the QP arising from triangulations of a marked
surface of genus $g$ with $b$ boundary components and a marked point on
each component, and $\cT_{g,b}$ denotes the class of the corresponding
(gentle) Jacobian algebras.

\begin{lemma} \label{l:Qnet}
A quiver in $\cQ_{g,b}$ has $n$ vertices, $e$ arrows and $t$ triangles,
where
\begin{align*}
n = 6(g-1)+4b &,& e=12(g-1)+7b &,& t=4(g-1)+2b .
\end{align*}
\end{lemma}
\begin{proof}
The claim on the number of vertices follows from~\cite{FST08}. Since
each boundary component contains exactly one marked point, any
triangulation consists of $t$ internal triangles and $b$ non-internal
triangles, one for each boundary component (which becomes one of its
sides).

Fix a triangulation corresponding to the quiver. We count in two ways
the pairs $(\gamma,\Delta)$ where $\Delta$ a triangle and $\gamma$ is
an arc which is one of its sides. On the one hand, for every internal
triangle $\Delta$ there are $3$ such arcs $\gamma$ and for every
non-internal one there are $2$ such arcs, giving us a total of $3t+2b$
pairs. On the other hand, each arc is a side of exactly two triangles
so that the total number of such pairs is $2n$. From the equality
$2n=3t+2b$ we deduce the formula for $t$.

Finally, each internal triangle gives rise to a $3$-cycle in the quiver
and hence to three arrows whereas each non-internal one gives rise to
one arrow. Thus $e = 3t+b$ and we get the formula for $e$ as well.
\end{proof}

For a gentle algebra $\gL$, denote its Avella-Alaminos-Geiss derived
invariant~\cite{AvellaAlaminosGeiss08} by $\phi(\gL)$.

\begin{lemma} \label{l:AG}
Let $\gL \in \cT_{g,b}$. Then $\phi(\gL) = b(1,1)+t(0,3)$ where
$t=4(g-1)+2b$.
\end{lemma}
\begin{proof}
Let $C$ be a boundary component and $A$ the marked point on $C$. In the
triangulation corresponding to $\gL$, let $i_1, i_2, \dots, i_s$ denote
the arcs passing through $A$ traversed in an anti-clockwise order as in
the following picture.
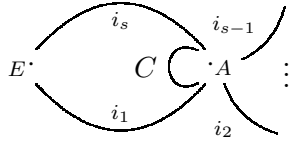
\[
\xymatrix@=1pc{
& & & & \\
{_E\cdot} \ar@{-}@/^2pc/[rrr]_{i_s} \ar@{-}@/_2pc/[rrr]^{i_1} & & {C} &
{\cdot_A} \ar@{-}@(ul,dl)[] \ar@{-}@/_/[dr]_{i_2}
\ar@{-}@/_/[ur]^{i_{s-1}} & {\vdots} \\
& & & &
}
\]

Then in the quiver with relations of $\gL$ the path $i_1 \to i_2 \to
\dots \to i_s$ is a maximal non-zero path whereas the arrow $i_1 \to
i_s$ induced by the non-internal triangle containing the arcs $i_1$,
$i_s$ is a maximal zero path.

Thus, each component $C$ contributes $(1,1)$ to $\phi(\gL)$. In
addition, each internal triangle in the triangulation yields an
oriented $3$-cycle with radical square zero relations, thus contributes
$(0,3)$ to $\phi(\gL)$.
\end{proof}

\begin{prop} \label{p:AGbt}
Let $\gL$ be a gentle algebra with $\phi(\gL) = b(1,1)+t(0,3)$ for some
$b>0$ and $t \geq 0$. Then:
\begin{enumerate}
\renewcommand{\theenumi}{\alph{enumi}}
\item \label{it:AGsurface}
All zero-relations lie in radical square zero oriented $3$-cycles.

\item \label{it:AGngbr}
For any vertex $v$, the subquiver formed by $v$ and the arrows incident
to $v$ is not one of the following:
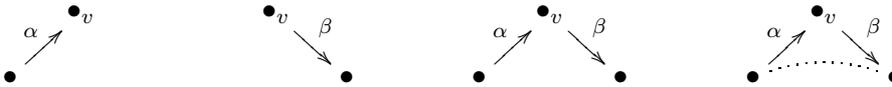
\begin{align*}
\xymatrix@=1pc{
& {\bullet_v} \\
{\bullet} \ar[ur]^{\alpha}
}
& &
\xymatrix@=1pc{
& {\bullet_v} \ar[dr]^{\beta} \\
& & {\bullet}
}
& &
\xymatrix@=1pc{
& {\bullet_v} \ar[dr]^{\beta} \\
{\bullet} \ar[ur]^{\alpha} & & {\bullet}
}
& &
\xymatrix@=1pc{
& {\bullet_v} \ar[dr]^{\beta} \\
{\bullet} \ar[ur]^{\alpha} \ar@{.}@/^/[rr] & & {\bullet}
}
\end{align*}
\end{enumerate}
\end{prop}
\begin{proof}
{\ }
\begin{enumerate}
\renewcommand{\theenumi}{\alph{enumi}}
\item
Any zero relation not in a radical square zero oriented cycle would
yield a contribution of some $(n,m)$ with $m>1$ to $\phi(\gL)$. In
addition, a radical square zero oriented cycle of length $m$ yields a
contribution of $(0,m)$.

\item
An inspection of each of the four cases reveals a contribution of
$(n,m)$ with $n>1$ to $\phi(\gL)$.
\end{enumerate}
\end{proof}

\begin{cor} \label{c:AGsurfacegood}
Let $\gL$ be a gentle algebra with $\phi(\gL) = b(1,1)+t(0,3)$ for some
$b>0$ and $t \geq 0$. Then $\gL \simeq \cP(Q,W)$ for some $(Q,W)$
arising from a triangulation of a marked unpunctured surface. Moreover,
all the mutations of $(Q,W)$ are good.
\end{cor}
\begin{proof}
The first assertion follows from Prop.~\ref{p:AGbt}(\ref{it:AGsurface})
and~\cite[Prop.~2.8]{ABCP10}. The second follows from
Prop.~\ref{p:AGbt}(\ref{it:AGngbr}) together with the characterization
of good mutations in Proposition~\ref{p:goodmut}.
\end{proof}

\begin{proof}[Proof of Theorem~\ref{t:gb}]
{\ }
\begin{enumerate}
\renewcommand{\theenumi}{\alph{enumi}}
\item
Since it is known that conditions $(\delta_1)$ and $(\delta_4)$ hold,
we only need to show condition $(\delta_3)$ which will then imply
$(\delta_2)$. Indeed, let $(Q,W) \in \cQ_{g,b}$. By Lemma~\ref{l:AG},
$\phi(\cP(Q,W)) = b(1,1)+t(0,3)$. Now the claim follows by
part~(\ref{it:AGngbr}) of Proposition~\ref{p:AGbt} together with
Proposition~\ref{p:goodmut}.

\item
Follows from Lemma~\ref{l:Qnet}.

\item
Follows from~\cite{Holm05} and Lemma~\ref{l:AG}.

\item
$\gL$ and $\gL'$ are connected by a sequence of QP mutations which are
good by property $(\delta_3)$ of $\cQ_{g,b}$ and hence induce the
required sequence of algebra mutations. Note that these are also
Brenner-Butler tilts according to Proposition~\ref{p:algmut}.

\item
Let $\gL'$ be derived equivalent to an algebra in $\cT_{g,b}$. By a
result of Schr\"{o}er and Zimmermann~\cite{SchroerZimmermann03}, the
class of gentle algebras is closed under derived equivalence and hence
$\gL'$ is also gentle with the same derived invariant $\phi(\gL') =
b(1,1)+t(0,3)$.

By Corollary~\ref{c:AGsurfacegood}, $\gL' \simeq \cP(Q,W)$ for some
$(Q,W)$ arising from a triangulation of a bordered unpunctured surface
$(S,M)$ and moreover all mutations are good. Since any mutation
$\mu_k(Q,W)$ leads to a derived equivalent Jacobian algebra, applying
again Corollary~\ref{c:AGsurfacegood} we see that all mutations of
$\mu_k(Q,W)$ are good as well. Applying mutations repeatedly we deduce
that the mutation class of $(Q,W)$ satisfies condition $(\delta_2)$,
and hence by the results of Section~\ref{ssec:necessary} $(S,M)$ is
either a disc with $4$ or $5$ marked points or a surface of genus $g'$
with $b'$ boundary components and a marked point on each component.
The first case is impossible since $\phi(KA_n)=(n+1,n-1)$~\cite[\S
7]{AvellaAlaminosGeiss08}. Thus we are in the second case and by
Lemma~\ref{l:AG} we must have $b'=b$ and $g'=g$ so that $\gL' \in
\cT_{g,b}$.

\item
Clear.
\end{enumerate}
\end{proof}

\begin{proof}[Proof of Theorem~\ref{t:SM}]
For the implications in one direction, combine Proposition~\ref{p:SMd2}
and Remark~\ref{rem:An}. The other direction follows from
Theorem~\ref{t:gb}.
\end{proof}

\subsection{The exceptional quivers}
Consider the quiver with potential $(Q_{p,q,r}, W_{p,q,r})$ of
Section~\ref{sec:exceptional}. By computing the Cartan matrix of the
Jacobian algebra $\cP(Q_{2,2,2}, W_{2,2,2})$ we see that its
determinant equals $4$, hence the same is true for any of the algebras
$\cP(Q_{p,q,r}, W_{p,q,r})$ obtained by gluing the linear quivers
$A_{p-1}$, $A_{q-1}$ and $A_{r-1}$.

If at least one of $p,q,r$ is greater than $3$, say $r>3$, then by
performing mutation at the vertex labeled $r-2$ we get the Jacobian
algebra $\cP(Q'_{p,q,r}, W'_{p,q,r})$ with the quiver
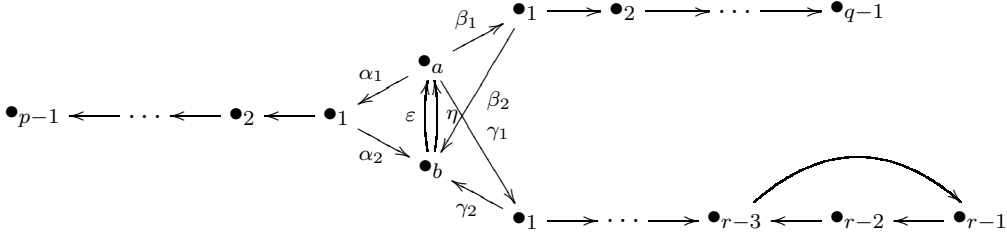
\[
\xymatrix@=0.5pc{
&& && && && && {\bullet_1}
\ar[dddll]^{\beta_2} \ar[rr] && {\bullet_2} \ar[rr]
&& {\dots} \ar[rr] && {\bullet_{q-1}} \\
&& && && && {\bullet_a}
\ar[urr]^{\beta_1} \ar[dddrr]^{\gamma_1} \ar[dll]_{\alpha_1} \\
{\bullet_{p-1}} && {\dots} \ar[ll] &&
{\bullet_2} \ar[ll] && {\bullet_1} \ar[ll] \ar[drr]_{\alpha_2} \\
&& && && && {\bullet_b} \ar@/^0.2pc/[uu]^{\eps} \ar@/_0.2pc/[uu]_{\eta} \\
&& && && && && {\bullet_1}
\ar[ull]^{\gamma_2} \ar[rr] && {\dots} \ar[rr]
&& {\bullet_{r-3}} \ar@/^2pc/[rrrr] && {\bullet_{r-2}} \ar[ll]
&& {\bullet_{r-1}} \ar[ll]
}
\]
and potential $W'_{p,q,r} = W_{p,q,r} + \Delta$, where $\Delta$ is the
new $3$-cycle in $Q'_{p,q,r}$. The Cartan matrix of this algebra has
determinant $8$, as can be seen by direct calculation or
invoking~\cite{Ladkani11b}. It follows that $\cP(Q_{p,q,r},W_{p,q,r})$
and $\cP(Q'_{p,q,r}, W'_{p,q,r})$ are not derived equivalent and hence
$\cQ_{p,q,r}$ does not satisfy condition $(\delta_2)$.

We are left with the case where $p,q,r \leq 3$. The mutation classes
$\cQ_{2,2,2}$, $\cQ_{2,2,3}$ and $\cQ_{2,3,3}$ coincide with the
mutation classes of the acyclic quivers $\tilde{D}_4$, $\tilde{D}_5$
and $\tilde{E}_6$, respectively, hence they cannot satisfy condition
$(\delta_2)$ either.

For the class $\cQ_{3,3,3}$ as well as for the mutation class of $(X_6,
W_6)$, one computes the Jacobian algebras in these finite mutation
classes and applies the algorithms developed in~\cite{Ladkani10}.

\section{Quivers}
\label{sec:quivers}

In this section we provide a recipe to produce explicit quivers in each
of the classes $\cQ_{g,b}$ introduced in Section~\ref{sec:surface}.

\subsection{The case $g>0$ and $b=1$}

We draw the fundamental polygon with $4g$ edges labeled
$1,2,1,2,\dots,2g-1,2g,2g-1,2g$ corresponding to a surface of genus
$g$, and put the single hole inside the polygon near one of its
vertices which serves as the marked point (recall that all vertices of
the polygon get identified on the surface). Then there exist
triangulations as shown in Figure~\ref{fig:Tg1}.

\begin{figure}
\begin{align*}
\xymatrix@=1pc{
& {\circ} \ar@{-}[dr]^1 \ar@{-}@/^0.5pc/[dd]^x \ar@{-}@/_0.5pc/[dd]_y \\
{\cdot} \ar@{-}[ur]^2 & & {\cdot} \ar@{-}[dl]^2 \\
& {\cdot} \ar@{-}[ul]^1
}
& &
\xymatrix@=0.3pc{
& && {\circ}
\ar@{-}[drr]^1 \ar@{-}[dddrrr] \ar@{-}[dddddrr] \ar@{-}@/^0.3pc/[dddddd]^x
\ar@{-}[dddlll] \ar@{-}[dddddll] \ar@{-}@/_0.3pc/[dddddd]_y \\
& {\cdot} \ar@{-}[urr]^4 && && {\cdot} \ar@{-}[ddr]^2 \\ \\
{\cdot} \ar@{-}[uur]^3 && && && {\cdot} \ar@{-}[ddl]^1 \\ \\
& {\cdot} \ar@{-}[uul]^4 && && {\cdot} \ar@{-}[dll]^2 \\
& && {\cdot} \ar@{-}[ull]^3
}
& &
\xymatrix@=0.2pc{
& && && {\circ}
\ar@{-}[dddrrrr] \ar@{-}[dddddrrrrr] \ar@{-}[dddddddrrrr]
\ar@{-}[dddddddddrr] \ar@{-}@/^0.2pc/[dddddddddd]^x
\ar@{-}[dddllll] \ar@{-}[dddddlllll] \ar@{-}[dddddddllll]
\ar@{-}[dddddddddll] \ar@{-}@/_0.2pc/[dddddddddd]_y \\
& && {\cdot} \ar@{-}[urr]^6 && && {\cdot} \ar@{-}[ull]_1 \\ \\
& {\cdot} \ar@{-}[uurr]^5 && && && && {\cdot} \ar@{-}[uull]_2 \\ \\
{\cdot} \ar@{-}[uur]^6 & && && && && & {\cdot} \ar@{-}[uul]_1 \\ \\
& {\cdot} \ar@{-}[uul]^5 && && && && {\cdot} \ar@{-}[uur]_2 \\ \\
& && {\cdot} \ar@{-}[uull]^4 && && {\cdot} \ar@{-}[uurr]_3 \\
& && && {\cdot} \ar@{-}[ull]^3 \ar@{-}[urr]_4
}
\end{align*}
\caption{Triangulations of bordered surfaces of genus $g$ with one
hole, denoted by $\circ$, for $g=1,2,3$. Edges having the same label
are identified.}
\label{fig:Tg1}
\end{figure}
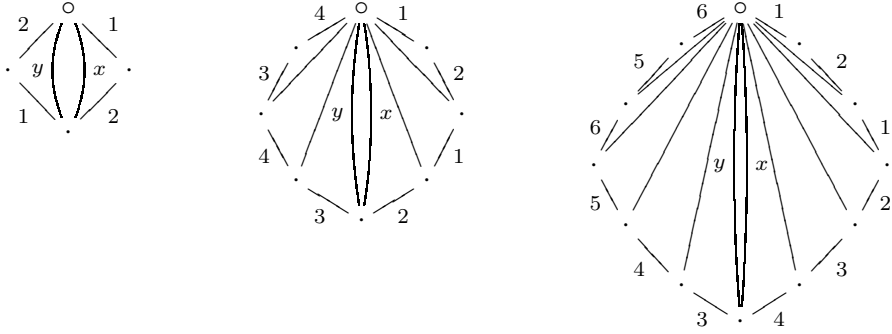

The corresponding quivers are then built by gluing three kinds of
building blocks shown in Figure~\ref{fig:blocks}, as demonstrated in
Figure~\ref{fig:Qg1} for $g \leq 4$. Each block corresponds to a pair
of consecutive labels $2i-1,2i$ of edges in the polygon, and it is
constructed by considering all the triangles adjacent to these edges.

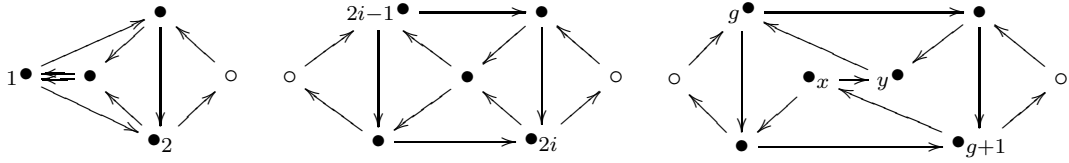
\begin{figure}
\begin{align*}
\xymatrix@=1pc{
& & {\bullet} \ar[dl] \ar[dd] \\
{_1\bullet} \ar[urr] \ar[drr] & {\bullet} \ar@/^0.1pc/[l] \ar@/_0.1pc/[l]
& & {\circ} \ar[ul] \\
& & {\bullet_2} \ar[ul] \ar[ur]
}
& &
\xymatrix@=1pc{
& {_{2i-1}\bullet} \ar[rr] \ar[dd] & & {\bullet} \ar[dl] \ar[dd] \\
{\circ} \ar[ur] & & {\bullet} \ar[ul] \ar[dl] & & {\circ} \ar[ul] \\
& {\bullet} \ar[ul] \ar[rr] & & {\bullet_{2i}} \ar[ul] \ar[ur]
}
& &
\xymatrix@R=0.9pc@C=0.9pc{
& {_g\bullet} \ar[rrr] \ar[dd] & & & {\bullet} \ar[dl] \ar[dd] \\
{\circ} \ar[ur] & & {\bullet_x} \ar[dl] \ar[r] & {_y\bullet} \ar[ull]
& & {\circ} \ar[ul] \\
& {\bullet} \ar[ul] \ar[rrr] & & & {\bullet_{g+1}} \ar[ull] \ar[ur]
}
\end{align*}
\caption{Three building blocks for a quiver in $\cQ_{g,1}$. Gluing
points are marked with $\circ$.} \label{fig:blocks}
\end{figure}

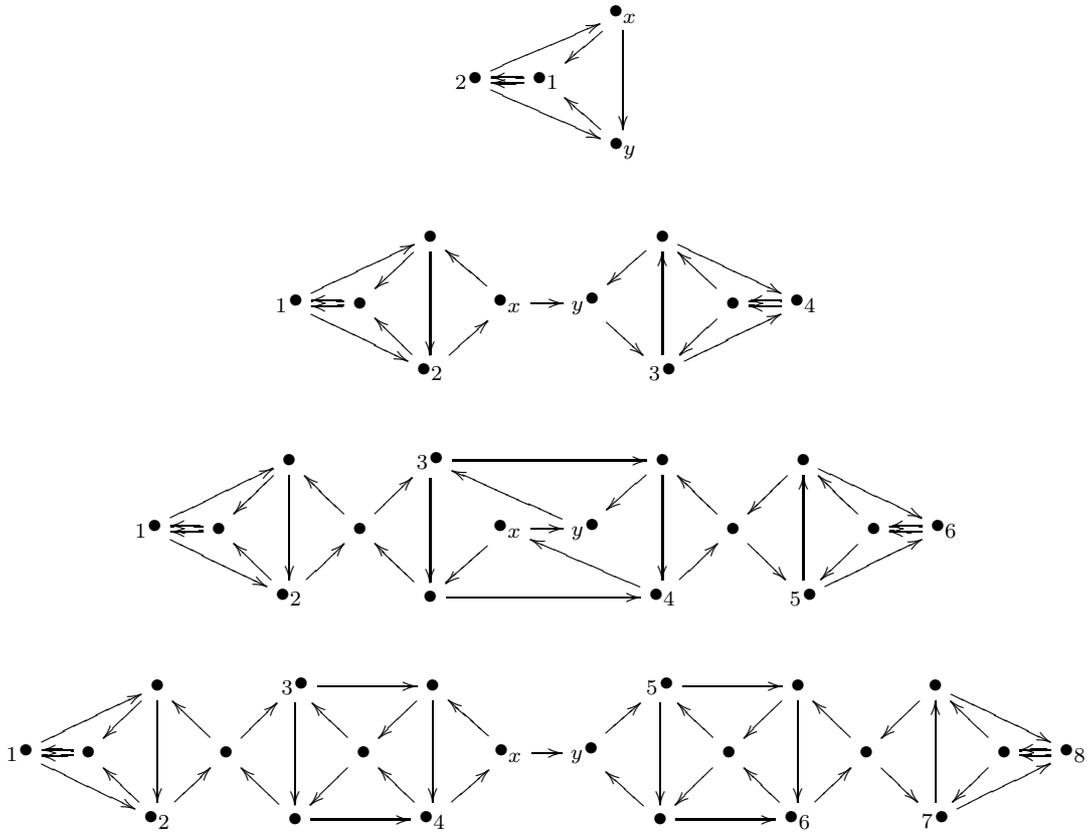
\begin{figure}
\begin{gather*}
\xymatrix@=1pc{
& & {\bullet_x} \ar[dl] \ar[dd] \\
{_2\bullet} \ar[urr] \ar[drr] & {\bullet_1} \ar@/^0.1pc/[l] \ar@/_0.1pc/[l] \\
& & {\bullet_y} \ar[ul]
}
\\ \\
\xymatrix@=1pc{
& & {\bullet} \ar[dl] \ar[dd] & & & {\bullet} \ar[dl] \ar[drr] \\
{_1\bullet} \ar[urr] \ar[drr] & {\bullet} \ar@/^0.1pc/[l] \ar@/_0.1pc/[l]
& & {\bullet_x} \ar[ul] \ar[r] & {_y\bullet} \ar[dr]
& & {\bullet} \ar[ul] \ar[dl] & {\bullet_4} \ar@/^0.1pc/[l] \ar@/_0.1pc/[l] \\
& & {\bullet_2} \ar[ul] \ar[ur] & & & {_3\bullet} \ar[uu] \ar[urr]
}
\\ \\
\xymatrix@=1pc{
& & {\bullet} \ar[dl] \ar[dd]
& & {_3\bullet} \ar[rrr] \ar[dd] & & & {\bullet} \ar[dl] \ar[dd]
& & {\bullet} \ar[dl] \ar[drr] \\
{_1\bullet} \ar[urr] \ar[drr] & {\bullet} \ar@/^0.1pc/[l] \ar@/_0.1pc/[l]
& & {\bullet} \ar[ul] \ar[ur]
& & {\bullet_x} \ar[dl] \ar[r] & {_y\bullet} \ar[ull]
& & {\bullet} \ar[ul] \ar[dr] & & {\bullet} \ar[ul] \ar[dl]
& {\bullet_6} \ar@/^0.1pc/[l] \ar@/_0.1pc/[l] \\
& & {\bullet_2} \ar[ul] \ar[ur]
& & {\bullet} \ar[ul] \ar[rrr] & & & {\bullet_4} \ar[ull] \ar[ur]
& & {_5\bullet} \ar[uu] \ar[urr]
}
\\ \\
\xymatrix@=0.95pc{
& & {\bullet} \ar[dl] \ar[dd]
& & {_3\bullet} \ar[rr] \ar[dd] & & {\bullet} \ar[dl] \ar[dd]
& & & {_5\bullet} \ar[rr] \ar[dd] & & {\bullet} \ar[dl] \ar[dd]
& & {\bullet} \ar[dl] \ar[drr] \\
{_1\bullet} \ar[urr] \ar[drr] & {\bullet} \ar@/^0.1pc/[l] \ar@/_0.1pc/[l]
& & {\bullet} \ar[ul] \ar[ur] & & {\bullet} \ar[ul] \ar[dl]
& & {\bullet_x} \ar[ul] \ar[r] & {_y\bullet} \ar[ur]
& & {\bullet} \ar[ul] \ar[dl] & & {\bullet} \ar[ul] \ar[dr]
& & {\bullet} \ar[ul] \ar[dl] & {\bullet_8} \ar@/^0.1pc/[l] \ar@/_0.1pc/[l] \\
& & {\bullet_2} \ar[ul] \ar[ur]
& & {\bullet} \ar[ul] \ar[rr] & & {\bullet_4} \ar[ul] \ar[ur]
& & & {\bullet} \ar[ul] \ar[rr] & & {\bullet_6} \ar[ul] \ar[ur]
& & {_7\bullet} \ar[uu] \ar[urr]
}
\end{gather*}
\caption{Quivers in $\cQ_{g,1}$ for $g=1,2,3,4$.}
\label{fig:Qg1}
\end{figure}

The left block corresponds to the initial and terminal pairs of labels
$\{1,2\}$ and $\{2g-1,2g\}$. We have drawn the picture only for one
pair, as it is symmetric (and isomorphic) for the other pair. The
middle one corresponds to all the intermediate pairs $\{2i-1,2i\}$
where $1 < i < g$ and $g \neq 2i-1$, whereas the right one arises from
the middle pair $\{g,g+1\}$ when $g$ is odd and involves the non-inner
triangle containing the arcs $x$ and $y$.

\subsection{The case $g>0$ and $b>1$}
The only triangle in the above triangulations which is not inner is the
``middle'' one
\begin{equation}
\label{e:g1}
\xymatrix{
{_A{\circ}} \ar@{-}@/^1pc/[r]^x \ar@{-}@/_1pc/[r]_y & {\cdot}_A
}
\end{equation}
where the marked point is indicated by the letter $A$ and $\circ$ is
the hole.

When there is more than one hole, we may arrange the other holes inside
this triangle and refine the triangulation to pass through the
additional marked points $B, C, \dots$ as in the following pictures
\begin{align}
\label{e:g23}
\xymatrix{
{_A{\circ}} \ar@{-}@/^0.5pc/[r] \ar@{-}@/_0.5pc/[r]
\ar@{-}@/^1pc/[rr]^x \ar@{-}@/_1pc/[rr]_y &
{_B{\circ}} \ar@{-}@/^0.5pc/[r] \ar@{-}@/_0.5pc/[r] &
{\cdot_A}
}
& &
\xymatrix{
{_A{\circ}} \ar@{-}@/^0.5pc/[r] \ar@{-}@/_0.5pc/[r]
\ar@{-}@/^1pc/[rr] \ar@{-}@/_1pc/[rr]
\ar@{-}@/^1.5pc/[rrr]^x \ar@{-}@/_1.5pc/[rrr]_y &
{_B{\circ}} \ar@{-}@/^0.5pc/[r] \ar@{-}@/_0.5pc/[r] &
{_C{\circ}} \ar@{-}@/^0.5pc/[r] \ar@{-}@/_0.5pc/[r] &
{\cdot_A}
}
\end{align}

A quiver in $\cQ_{g,b}$ is thus obtained from our representative in
$\cQ_{g,1}$ by replacing the single arrow $\xymatrix{{\bullet_x} \ar[r]
& {\bullet_y}}$ corresponding to the picture~\eqref{e:g1} by the quiver
\begin{equation}
\label{e:Qb}
\xymatrix@=1pc{
{_x\bullet} \ar[dr] & & {\bullet} \ar[ll] \ar[dr] & &
{\bullet} \ar[ll] & {\dots} & {\bullet} \ar[dr] & &
{\bullet} \ar[ll] \ar[ddd] \\
& {\bullet} \ar[ur] \ar[d] & & {\bullet} \ar[ur] \ar[d]
& & & & {\bullet} \ar[ur] \ar[d] \\
& {\bullet} \ar[dl] & & {\bullet} \ar[dl] & & & & {\bullet} \ar[dl] \\
{_y\bullet} \ar[rr] & & {\bullet} \ar[rr] \ar[ul] & &
{\bullet} \ar[ul] & {\ldots} & {\bullet} \ar[rr]
& & {\bullet} \ar[ul]
}
\end{equation}
with $2(b-1)$ oriented triangles. The $b$ vertical arrows of this
quiver correspond to the $b$ triangles which are not inner in the
triangulation.

\subsection{The case $g=0$ and $b>1$}

This case is quite similar to the previous one. In fact, by replacing
the copy of the marked point $\cdot_A$ in the pictures~\eqref{e:g1}
and~\eqref{e:g23} by an additional hole we get triangulations of the
sphere with $b$ holes and $b$ marked points, one at each hole, which
are shown below for $b=2,3,4$.
\begin{align*}
\xymatrix{
{_A{\circ}} \ar@{-}@/^1pc/[r]^x \ar@{-}@/_1pc/[r]_y & {_B{\circ}}
}
& &
\xymatrix{
{_A{\circ}} \ar@{-}@/^0.5pc/[r] \ar@{-}@/_0.5pc/[r]
\ar@{-}@/^1pc/[rr]^x \ar@{-}@/_1pc/[rr]_y &
{_B{\circ}} \ar@{-}@/^0.5pc/[r] \ar@{-}@/_0.5pc/[r] &
{_C{\circ}}
}
& &
\xymatrix{
{_A{\circ}} \ar@{-}@/^0.5pc/[r] \ar@{-}@/_0.5pc/[r]
\ar@{-}@/^1pc/[rr] \ar@{-}@/_1pc/[rr]
\ar@{-}@/^1.5pc/[rrr]^x \ar@{-}@/_1.5pc/[rrr]_y &
{_B{\circ}} \ar@{-}@/^0.5pc/[r] \ar@{-}@/_0.5pc/[r] &
{_C{\circ}} \ar@{-}@/^0.5pc/[r] \ar@{-}@/_0.5pc/[r] &
{_D{\circ}}
}
\end{align*}

The corresponding quiver in $\cQ_{0,b}$ is obtained from the one
in~\eqref{e:Qb} with $2(b-2)$ oriented triangles by adding an arrow
from $x$ to $y$ coming from the ``external'' triangle consisting of the
edges $x$, $y$ and the boundary of the rightmost hole, see for example
the top row of Figure~\ref{fig:Qgb}.

\bibliographystyle{amsplain}
\bibliography{surface}

\end{document}